\documentclass[12pt]{article}

\usepackage{amsmath}
\usepackage{amssymb}
\usepackage{amsthm}
\usepackage{graphicx}
\usepackage[sort]{natbib}
\usepackage{color}
\usepackage{verbatim}
\usepackage{url}
\usepackage{mathtools}
\usepackage{underoverlap}

\newcommand{\w}{w}
\newcommand{\f}{u}
\newcommand{\z}{v}
\newcommand{\id}{h}

\newcommand{\A}{\mathbf{A}}
\newcommand{\BB}{\mathbf{B}}
\newcommand{\C}{\mathbf{C}}
\newcommand{\D}{\mathbf{D}}
\newcommand{\II}{\mathbf{I}}
\newcommand{\bb}{\mathbf{b}}
\newcommand{\bc}{\mathbf{c}}
\newcommand{\be}{\mathbf{e}}
\newcommand{\bt}{\mathbf{t}}
\newcommand{\bx}{\mathbf{x}}
\newcommand{\by}{\mathbf{y}}
\newcommand{\bn}{\mathbf{n}}

\newcommand{\rf}{A}

\newcommand{\G}{\mathcal{G}}            

\newtheorem{theorem}{Theorem}[section]
\newtheorem{corollary}[theorem]{Corollary}

\begin{document}

\title{Joint distribution of rises,  falls, and  number of runs in random sequences}

\author{Yong Kong\\
School of Public Health\\
Yale University\\
300 Cedar Street, New Haven, CT 06520, USA\\
email: \texttt{yong.kong@yale.edu} }

\date{}

\maketitle

\newpage

\begin{abstract}
  By using the matrix formulation of the two-step approach to the
  distributions of runs, a recursive relation
  and an explicit expression are derived for
  the generating function of
  the joint distribution of rises and falls for multivariate random sequences
  in terms of generating functions of individual letters,
  from which the generating functions
  of the joint distribution of rises,  falls, and  number of runs
  are obtained. 
  An explicit formula for the joint distribution
  of rises and falls with arbitrary specification is also obtained.
\end{abstract}

%
%
  
\newpage

\section{Introduction} \label{S:intro}

For a sequence $\sigma = ( a_1, a_2, \dots, a_n )$ where 
$a_i \in \{1, 2, \dots, k\}$, a pair of consecutive elements $a_i$ and $a_{i+1}$
is called a \emph{rise} when $a_i < a_{i+1}$,
a \emph{fall} when $a_i > a_{i+1}$,
and a \emph{level} when $a_i = a_{i+1}$.
If $r$, $s$, and $l$ denote the number of rises, falls, and
levels, they satisfy the obvious relation:
$
 r + s + l = n - 1.
$
 
A \emph{run} in $\sigma$ is defined
as a stretch of consecutive identical integers separated by other 
integers or the end of the sequence
\citep{Balakrishnan2002,Mood1940}.
If integer $i$ appears $n_i$ times in $\sigma$, 
we call $\sigma$ of \emph{specification} $\bn=[n_1, n_2, \dots, n_k]$,
with $\sum_{i=1}^k n_i = n$.
If $b$ denotes the number of runs, it's easy to see that
$b + l =n$, and $r+s+1=b$.
If integer $i$ has $b_i$ runs,
then $\sum_{i=1}^k b_i = b$.
We are interested in the number $A(\bn; r, s, \bb)$ of sequences with
exactly $r$ rises, $s$ falls, and numbers of runs
$\bb = [b_1, \dots, b_k]$ for a given specification $\bn$.
In Table~\ref{Tb:x2x3} the $10$ permutations of
specification $\bn=[2, 3]$ are listed.

When $\bn = [1,1, \dots, 1]$ and only rises are considered,
$A([1,1,, \dots, 1]; r)$ is the well-known Eulerian number $A(k, r)$
\citep{Graham:1994:CMF:562056}.
Simon Newcomb's problem is a generalization of the Eulerian number
with arbitrary specification
\citep{dillon1969}.
The explicit formula for this problem is given by
\citep{macmahon1915,dillon1969} 
\[
A(\bn; r) =
  \sum_{j=0}^r (-1)^j \binom{n+1}{j} \prod_{i=1}^k \binom{n_i + r -j}{n_i} .
\]
The joint distribution of rise and falls with arbitrary specification was studied by  
different researchers
~\citep{carlitz1972,jackson1977,reilly1979,Fu2000}.
Note that these researchers, except in
\cite{reilly1979},  included an additional
rise at the beginning and an additional fall at the end of each sequence.
For this problem, as Carlitz pointed out, no simple explicit formulas have been
found
~\citep{carlitz1972,carlitz1975}.
In Theorem~\ref{T:ers} of the present paper we give an explicit formula
for $A(\bn; r, s)$,
the joint distribution of rise and falls with arbitrary specification $\bn$.

%
%
%

%
%
%

%
%
%

In this paper we use the matrix formulation of the
two-step approach to distributions of runs
~\citep{Kong2006,Kong2013,Kong2014,Kong2015,Kong2016}
to obtain the generating function $\G_k$ for the whole system
in terms of the generating function $g_i$ of individual integers ($i=0, \dots, k$).
The advantage of expressing $\G_k$ in terms of $g_i$ is that,
by using appropriate $g_i$'s,
various joint distributions can be obtained easily.
The previously obtained results are
straightforward specializations,
and new results of joint distribution of rises, falls,
and number of runs are obtained similarly.
In Theorem~\ref{T:rGF} we first obtain a recursive relation for
$\G_k$,
and in Theorem~\ref{T:eGF} the explicit formula of $\G_k$
is obtained in terms of $g_i$.
The joint  distributions of rises, falls, and number of runs
are obtained in Corollary~\ref{C:xt}, Corollary~\ref{C:xtz},  and Corollary~\ref{C:xtiz} .
In Theorem~\ref{T:ers} an explicit formula
is derived for 
the joint distribution of rises and falls.

In the following, let
\[
\G_{k} = \sum_{\bn, r, s, \cdots}  \rf(\bn; r, s, \cdots) \w^r \f^s \cdots \bx^\bn
\]
where
$\bx = [x_1, x_2, \dots, x_k]$ and $\bx^\bn = x_1^{n_1} \cdots x_k^{n_k}$.
We use the notation $[x^n] f(x)$ to denote the coefficient of
$x^n$ in the series expansion of $f(x)$.

\begin{table} \label{Tb:x2x3}
  \centering
  \caption{The $10$ permutations of multiset $\{1,1,2,2,2\}$
    for $k=2$, $n_1=2$, and $n_2=3$.
  }
  \begin{tabular}{ cccccccc }
    \hline \hline
     & $r$ & $s$ & $l$ & $b$ & $b_1$ & $b_2$ & $\w^r \f^s \z^l t_1^{b_1}t_2^{b_2} $\\
     \hline
  $1 \enspace \UOLoverline{1 \enspace 2} \enspace 2 \enspace 2$ & $1$ & $0$ & $3$ & $2$ & $1$ & $1$ & $\w \z^3 t_1 t_2 $       \\
  ${}\UOLoverline{1 \enspace}[2]  \UOLunderline{\enspace}[1]\UOLoverline{\enspace 2} \enspace 2$ & $2$ & $1$ & $1$ & $4$ & $2$ & $2$ & $\w^2 \f \z t_1^2 t_2^2 $ \\
  ${}\UOLoverline{1 \enspace 2} \enspace \UOLunderline{2 \enspace}[1]\UOLoverline{\enspace 2}$   & $2$ & $1$ & $1$ & $4$ & $2$ & $2$ & $\w^2 \f \z t_1^2 t_2^2 $  \\
  ${}\UOLoverline{1 \enspace 2} \enspace 2 \enspace \UOLunderline{2 \enspace 1}$  & $1$ & $1$ & $2$ & $3$ & $2$ & $1$ & $\w \f \z^2 t_1^2 t_2$    \\
  ${}\UOLunderline{2 \enspace 1} \enspace \UOLoverline{1 \enspace  2} \enspace 2$ & $1$ & $1$ & $2$ & $3$ & $1$ & $2$ & $\w \f \z^2 t_1 t_2^2$    \\
  ${}\UOLunderline{2 \enspace}[1] \UOLoverline{\enspace }[2] \UOLunderline{\enspace}[1] \UOLoverline{\enspace 2}$ & $2$ & $2$ & $0$ & $5$ & $2$ & $3$ & $\w^2 \f^2 t_1^2 t_2^3$  \\
  ${}\UOLunderline{2 \enspace}[1] \UOLoverline{\enspace 2} \enspace \UOLunderline{2 \enspace 1}$ & $1$ & $2$  & $1$ & $4$ & $2$ & $2$ & $\w \f^2 \z t_1^2 t_2^2$    \\
  ${}2 \enspace \UOLunderline{2 \enspace 1} \enspace \UOLoverline{1 \enspace  2}$ & $1$ & $1$ & $2$ & $3$ & $1$ & $2$ & $\w \f \z^2 t_1 t_2^2$    \\
  $2 \enspace \UOLunderline{2 \enspace}[1] \UOLoverline{\enspace }[2] \UOLunderline{\enspace 1} $ & $1$ & $2$ & $1$ & $4$ & $2$ & $2$ & $\w \f^2 \z t_1^2 t_2^2 $   \\
  $2 \enspace 2 \enspace \UOLunderline{2 \enspace 1} \enspace 1 $ & $0$ & $1$ & $3$ & $2$ & $1$ & $1$ & $\f \z^3 t_1 t_2$ \\
\hline
\end{tabular}
\end{table}


\section{Recursive equation for $\G_{k}$}

\begin{theorem}[Recursive equation for $\G_{k}$] \label{T:rGF}
  The generating function $\G_{k}$ satisfies the recursive equation
\begin{equation} \label{E:rec}
\G_{k} = \G_{k-1} + \dfrac{\displaystyle g_k  (1 + \w \G_{k-1}) (1 + \f\G_{k-1})}
                        {\displaystyle 1 - \w\f g_k \G_{k-1}}
\end{equation}
with $\G_{1} = g_1$.
\end{theorem}

\begin{proof}
  Based on the matrix formulation of the two-step approach to the
  distributions of runs
  \citep{Kong2006,Kong2016},
  $\G_{k}$ is given by
\begin{equation} \label{E:matrix}
  \G_k = \be_k \mathbf{M}_k^{-1} \by_k,
\end{equation}
where 
\begin{equation*} \label{E:ek}
   \be_k = [\underbrace{1, 1, \cdots, 1}_k],
\end{equation*}
\begin{equation*} \label{E:fk}
   \by_k = [g_1, g_2, \cdots, g_k]^T,
\end{equation*}
and
\begin{equation} \label{E:M}
 \mathbf{M}_k =   
  \begin{bmatrix}
           1 & -g_1 \w  & -g_1 \w  & \cdots & -g_1 \w\\
 -g_2 \f &  1           & -g_2 \w  & \cdots & -g_2 \w \\
      \vdots & \vdots       & \vdots       & \vdots & \vdots      \\
 -g_{k-1} \f & \cdots & -g_{k-1} \f  & 1        & -g_{k-1} \w \\
 -g_k \f    & \cdots & -g_k \f     & -g_k \f  & 1 
   \end{bmatrix}.
\end{equation}
Here the indeterminates $\w$ and $\f$ are used to track the rises and falls, respectively,
and $g_i$'s are the generating function for the $i$th integer.
Eq.~\eqref{E:M} can be written in the following block form:
\[
\mathbf{M}_k = \begin{bmatrix}
  \mathbf{M}_{k-1}   & \bb \\
  \bc          & 1
  \end{bmatrix},
\]
where
\[
 \bb = [-g_1 \w, -g_2 \w, \cdots, -g_{k-1} \w]^T = -\w \mathbf{y}_{k-1},
\]
and
\[
 \bc = [-g_k \f, -g_k \f, \cdots, -g_k \f] = -\f g_k \mathbf{e}_{k-1} .
 \]

By applying the following matrix identity to Eq.~\eqref{E:matrix},
\[
\begin{bmatrix}
  \A         & \mathbf{b} \\
  \mathbf{c} & d  
\end{bmatrix}^{-1}
= \begin{bmatrix}
  \A^{-1} + \frac{1}{\id} \A^{-1} \bb \bc \A^{-1} & -\frac{1}{\id} \A^{-1} \bb \\
  -\frac{1}{\id} \bc  \A^{-1} & \frac{1}{\id}
  \end{bmatrix}
\]
where
\[
 \id = d - \bc \A^{-1} \bb ,
 \]
the recursive equation Eq.~\eqref{E:rec} is obtained: 
\begin{align*}
\G_{k} &=  \mathbf{e}_k \mathbf{M}_k^{-1} \mathbf{y}_k
= [\be_{k-1}, 1]
\begin{bmatrix}
  \A^{-1} + \frac{1}{\id} \A^{-1} \bb \bc \A^{-1} & -\frac{1}{\id} \A^{-1} \bb \\
  -\frac{1}{\id} \bc  \A^{-1} & \frac{1}{\id}
\end{bmatrix}
\begin{bmatrix}
  \by_{k-1} \\
  g_k
\end{bmatrix}\\
&=\G_{k-1} + \dfrac{\displaystyle g_k  (1 + \w \G_{k-1}) (1 + \f\G_{k-1})}
                        {\displaystyle 1 - \w\f g_k \G_{k-1}} .
\end{align*}

\end{proof}

\section{Explicit formula for $\G_{k}$
and generating functions of joint distributions of rises, falls, and runs}

\begin{theorem}[Explicit formula for $\G_{k}$] \label{T:eGF}
  The explicit formula for $\G_{k}$ in terms of the generating functions
  $g_i$ is given by
\begin{equation} \label{E:eGF}
\G_{k} = \dfrac{ \prod_{i=1}^k (1 + g_i \f) - \prod_{i=1}^k (1 + g_i \w)}
  {  \f\prod_{i=1}^k (1 + g_i \w) - \w\prod_{i=1}^k (1 + g_i \f) }
  = \frac{P_{k} - Q_{k}}{\f Q_{k} - \w P_{k}},
\end{equation}
where $P_{k} = \prod_{i=1}^k (1 + g_i \f)$
  and $Q_{k} = \prod_{i=1}^k (1 + g_i \w)$.
\end{theorem}                        

\begin{proof}
  By induction using Eq.~\eqref{E:rec} in Theorem~\ref{T:rGF}.
\end{proof}

From Theorem~\ref{T:eGF} generating functions of various joint distributions
can be obtained by using appropriate $g_i$'s.
Some examples are:

\begin{corollary} \label{C:x}
  The explicit formula for $\G_{k}$ of joint distribution of
  rises and falls in terms of $x_i$ is given by
  \begin{equation} \label{E:xGF}
    \begin{aligned}
      \G_k &= \sum_{\bn, r, s}  \rf(\bn; r, s) \w^r \f^s \bx^\bn \\
      &= \dfrac{ \prod_{i=1}^k \left[ 1 - x_i(1-\f)\right] - \prod_{i=1}^k \left[ 1-x_i (1-\w)\right] }
  { \f \prod_{i=1}^k \left[ 1-x_i (1-\w)\right] - \w \prod_{i=1}^k \left[ 1 - x_i(1-\f)\right] } .
\end{aligned}
    \end{equation}
\end{corollary}
\begin{proof}
  Let
  \[
 g_i =  \sum_{j=1}  x_i^j =
\frac{x_i}{1-x_i} .
\]
\end{proof}
Eq.~\eqref{E:xGF} was obtained by ~\cite{carlitz1972}, ~\cite{jackson1977}
and ~\cite{reilly1979} using different methods.

\begin{corollary} \label{C:xt}
  The explicit formula for $\G_{k}$ for the joint distribution of
  rises, falls, and the total number of runs is given by
  \begin{equation} \label{E:xtGF}
\begin{aligned}
  \G_k &= \sum_{\bn, r, s, b}  \rf(\bn; r, s, b) \w^r \f^s t^b \bx^\bn \\
  &= \dfrac{ \prod_{i=1}^k \left[ 1 - x_i(1-\f t)\right] - \prod_{i=1}^k \left[ 1-x_i (1-\w t)\right] }
  { \f \prod_{i=1}^k \left[ 1-x_i (1-\w t)\right] - \w \prod_{i=1}^k \left[ 1 - x_i(1-\f t)\right] } .
\end{aligned}
\end{equation}
\end{corollary}
\begin{proof}
  Let
  \[
 g_i = t \sum_{j=1}  x_i^j =
\frac{x_i t}{1-x_i} .
\]
\end{proof}
As an example, from Eq.~\eqref{E:xtGF} we get for $k=2$, $n_1=2$, and $n_2=3$:
\[
  [x_1^2 x_2^3] \G_k = \sum_{r, s, b}  \rf(\bn; r, s, b) \w^r \f^s t^{b}
  =  \f t^2  + \w t^2 + 3 \w \f t^3 + 2 \w \f^2 t^4 + 2 \w^2 \f t^4  + \w^2 \f^2 t^5 ,
   \]
which is enumerated in Table~\ref{Tb:x2x3}.

\begin{corollary} \label{C:xtz}
  The explicit formula for $\G_{k}$ for the joint distribution of rises, falls, levels,
  and the total number of runs is given by
  \begin{equation} \label{E:xtzGF}
\begin{aligned}
  \G_k &= \sum_{\bn, r, s, l, b}  \rf(\bn; r, s, l, b) \w^r \f^s \z^l t^b \bx^\bn \\
  &= \dfrac{ \prod_{i=1}^k \left[ 1 - x_i(\z-\f t)\right] - \prod_{i=1}^k \left[ 1-x_i (\z-\w t)\right] }
  { \f \prod_{i=1}^k \left[ 1-x_i (\z-\w t)\right] - \w \prod_{i=1}^k \left[ 1 - x_i(\z-\f t)\right] } .
\end{aligned}
\end{equation}
\end{corollary}
\begin{proof}
  Let
  \[
   g_i = t \sum_{j=1}  x_i^j \z^{j-1} =
\frac{x_i t}{1-\z x_i}.
  \]
\end{proof}

\begin{corollary} \label{C:xtiz}
  The explicit formula for $\G_{k}$ for the joint distribution of rises, falls, levels,
  and the numbers of runs for each integer is given by
  \begin{equation} \label{E:xtizGF}
\begin{aligned}
  \G_k &= \sum_{\bn, r, s, l, \bb}  \rf(\bn; r, s, l, \bb) \w^r \f^s \z^l \bt^\bb \bx^\bn \\
  &= \dfrac{ \prod_{i=1}^k \left[ 1 - x_i(\z-\f t_i)\right] - \prod_{i=1}^k \left[ 1-x_i (\z-\w t_i)\right] }
  { \f \prod_{i=1}^k \left[ 1-x_i (\z-\w t_i)\right] - \w \prod_{i=1}^k \left[ 1 - x_i(\z-\f t_i)\right] } .
\end{aligned}
\end{equation}
\end{corollary}
\begin{proof}
  Let
  \[
   g_i = t_i \sum_{j=1}  x_i^j \z^{j-1} =
\frac{x_i t_i}{1-\z x_i}.
  \]
\end{proof}

For $k=2$, $n_1=2$, and $n_2=3$, Eq.~\eqref{E:xtizGF} gives
\begin{align*}
  [x_1^2 x_2^3] \G_k &= \sum_{r, s, \bb}  \rf(\bn; r, s, \bb) \w^r \f^s t_1^{b_1} t_2^{b_2} \\
  &=
  \f \z^3 t_1 t_2 + \w \z^3 t_1 t_2 + 2 \w \f \z^2 t_1 t_2^2 + \w \f \z^2 t_1^2 t_2 + 2 \w \f^2 \z t_1^2 t_2^2
  + 2\w^2 \f \z t_1^2 t_2^2 + \w^2 \f^2 t_1^2 t_2^3 ,
\end{align*}
as shown in Table~\ref{Tb:x2x3}.

\section{Explicit formula for $\rf(\bn;r,s)$, the number of rises and falls for a given specification }
   
From Eq.~\eqref{E:xGF} of Corollary~\ref{C:x} we obtain the explicit formula for
$\rf(\bn;r,s)$.
\begin{theorem}[Explicit formula for $\rf(\bn;r,s)$] \label{T:ers}
The explicit formula for $\rf(\bn;r,s)$ is given by 
\[
A(\bn; r, s) =
\sideset{}{'} \sum^{n_i}_{d_i = 0}
\sideset{}{'}\sum^{n_i}_{t_i = 0} 
\left[
\prod_{i=1}^k f(n_i, r-t+1, d_i, t_i)
-
\prod_{i=1}^k f(n_i, r-t, d_i, t_i)
\right]
\]
where the sum $\sum\nolimits'$ is taken for $d + t = r+s+1$
and $t \le r$
with $d = \sum_i d_i$ and $t = \sum_i t_i$.
The function $f$ is defined as 
\begin{equation} \label{E:f} 
f(n, m, d, t) =
\begin{dcases}
  \binom{n-1}{t-1}
  \binom{m + t - 1}{t} (-1)^t,  & d = 0,\\
  \frac{m}{d}
  \binom{n-1}{d+t-1}
  \binom{m + t -1}{d+t-1}
  \binom{d+t-1}{t}(-1)^t ,  & d \ne 0 .
  \end{dcases}
\end{equation}
\end{theorem}

\begin{proof}
  Let $P_k = \prod_{i=1}^k \left[ 1 - x_i(1-\f) \right]$ and
  $Q_k = \prod_{i=1}^k \left[ 1 - x_i(1-\w) \right]$.
  \begin{align*}
    \G_k &= \frac{P_{k} - Q_{k}}{\f Q_{k} - \w P_{k}}
    = \frac{P_{k} - Q_{k}}{\f Q_{k} \left[1 - \frac{\w P_{k}}{\f Q_{k} }  \right]}
    = \frac{P_{k} - Q_{k}}{\f Q_{k} }
    \sum_i \left[1 - \frac{\w P_{k}}{\f Q_{k} }  \right]^i \\
    &= \sum_i \frac{\w^i}{\f^{i+1}}
    \left[ \frac{P_{k}^{i+1}}{Q_{k}^{i+1}}  - \frac{P_{k}^{i}}{Q_{k}^{i}}  \right]
    = \sum_i \frac{\w^i}{\f^{i+1}}
    \left[
      \prod_{a=1}^k
      \frac{\left[ 1 - x_a(1-\f) \right]^{i+1}}
           {\left[ 1 - x_a(1-\w) \right]^{i+1}}
      - \prod_{a=1}^k
      \frac{\left[ 1 - x_a(1-\f) \right]^{i}}
           {\left[ 1 - x_a(1-\w) \right]^{i}} 
           \right]   .
\end{align*}
Each term in the last expression can be further expanded by binomial theorem,
  and after changing variables, we get
\begin{align*}
  A(\bn; r, s) &= \left[\w^r \f^s \prod_{i=1}^k x_i^{n_i} \right]  \G_k \\
  &=
  \sum_{j_i=0}^{n_i}
  \sideset{}{'}\sum_{d_i=0}^{j_i}
  \sideset{}{'}\sum_{t_i=0}^{n_i-j_i}
  \left[
\prod_{i=1}^k h(n_i, j_i, r-t+1, d_i, t_i)
-
\prod_{i=1}^k h(n_i, j_i, r-t, d_i, t_i)
\right],
\end{align*}
where  
  \[
  h(n, j, m, d, t) = 
  \binom{j}{d}
  \binom{n-j}{t}
  \binom{m}{j}
  \binom{n-j+m-1}{n-j}
  (-1)^{j+d+t} .
  \]
  By using Gosper-Zeilberger method
  ~\citep{AEQB},
  we can simplify the above expression by getting rid
  of the sums of $j_i$.
  Let
  \[
    f(n, m, d, t) = \sum_{j=0}^n h(n, j, m, l, t)
    \]
  be the sum of $h(n, j, m, d, t)$ with respect to $j$.  
  The Gosper-Zeilberger method constructs $G(n, j)$ as
  \[
   G(n, j) = h(n, j, m, d, t) \frac{(d-j)(n+m-j)}{n-t+1-j} ,
   \]
   which satisfies
   \[
     (n-d-t+1) h(n+1, j, m, l, t) - n h(n, j, m, l, t) =  G(n, j+1) - G(n, j).
   \]
  Summing with respect to $j$ leads to 
  the following linear recurrence equation satisfied by $f(n, m, d, t)$:  	
  \[
   (n-d-t+1) f(n+1, m, d, t)  - n f(n, m, d, t) = 0.
  \]
  By solving the recurrence we obtain the identity in Eq.~\eqref{E:f}. 
\end{proof}

 \newcommand{\noop}[1]{}

\end{document}